\documentclass[12pt,a4paper]{amsart}

\usepackage{mathtools,amsmath, amsfonts, amssymb, latexsym, amsthm,
  amscd, cancel, enumerate, rotating, comment, appendix, ulem,
  makecell, paralist,times,graphicx,pgf,tikz,tikz-cd,geometry,dirtytalk,shortvrb}
\usepackage[capitalise, noabbrev]{cleveref}
\usepackage{paralist}
\usepackage{graphicx}
\usepackage{pgf}
\usepackage{times}
\usepackage[all]{xy}
\usepackage{enumerate}
\usepackage[latin1]{inputenc}
\usepackage[english]{babel}
\usepackage[mathscr]{eucal}
\usepackage{xxcolor}
\usepackage{tikz-3dplot}
\usepackage{caption}
\usepackage{subcaption}
\usepackage[T1]{fontenc}      

\newtheorem{theorem}{Theorem}[section]
\newtheorem{lemma}[theorem]{Lemma}

\newtheorem{corollary}[theorem]{Corollary}

\newtheorem{question}[theorem]{Question}

\theoremstyle{definition}
\newtheorem{definition}[theorem]{Definition} 

\newtheorem{example}[theorem]{Example}

\newcommand{\qand}{\quad \mbox{and} \quad}

\newcommand{\qforall}{\quad \mbox{for all} \quad}

\newcommand{\qwhere}{\quad \mbox{where} \quad}
\newcommand{\qor}{\quad \mbox{or} \quad}
\newcommand{\st}{\colon}

\newcommand{\MP}{\mathcal{P}}
\newcommand{\A}{\mathcal{A}}
\newcommand{\bm}{\mathbf{m}}

\newcommand{\lk}{\mathrm{lk}}

\title{Cohen--Macaulay Squares of edge ideals}

\author[S.~Faridi]{Sara Faridi}\thanks{Faridi's research is supported by an NSERC Discovery Grant 2023-05929. Both authors are part of an AIM SQuaRE collaborative group. We are grateful to AIM for providing an opportunity for us to start this research.  The present paper was written while the second author was visiting Dalhousie University, March--April 2025.  We are grateful for the support and hospitality of  Dalhousie University.}
\address{(Sara Faridi) Department of Mathematics \& Statistics, Dalhousie University, 6297 Castine Way, PO BOX 15000, Halifax, NS, Canada B3H 4R2}
\email{faridi@dal.ca}

\author[T.~Hibi]{Takayuki Hibi}
\address{(Takayuki Hibi) Department of Pure and Applied Mathematics, Graduate School of Information Science and Technology, Osaka University, Suita, Osaka 565--0871, Japan}
\email{hibi@math.sci.osaka-u.ac.jp}

\begin{document}



\begin{abstract}
Let $G$ be a finite graph and $I(G)$ its edge ideal. We give a full description of the Stanley--Reisner complex of the polarization of $I(G)^2$, naturally  introducing the tools of  Stanley--Reisner theory in the study of the algebraic behaviour of powers of edge ideals.  As an application, we demonstrate how  Reisner's criterion can be applied directly to check if $I(G)^2$ is Cohen--Macaulay. We can show that if  $G$ belongs to the class of finite graphs which consists of cycles, whisker graphs, trees, connected chordal graphs and connected Cohen--Macaulay bipartite graphs, then the square $I(G)^2$ is Cohen--Macaulay if and only if either $G$ is the pentagon, the cycle of length $5$, or $G$ consists of exactly one edge.
\end{abstract}

\maketitle
\thispagestyle{empty}
\maketitle

\section*{introduction}
 
 Stanley--Reisner theory provides a powerful tool to study homological invariants of squarefree monomial ideals via simplicial homology. One of the main motivations of this paper is to extend the reach of Stanley--Reisner tools to the study of powers of squarefree monomial ideals. We take a step in this direction by considering a squarefree monomial ideal $I$ generated in degree $2$, that is, the {\it edge ideal} of a graph $G$. Edge ideals of graphs are studied heavily 
 from the point of view of the combinatorics of the graph. In Stanley--Reisner theory they are studied in the context of  {\it flag} simplicial complexes. From either point of view, the combinatorics of the underlying graph $G$ has a significant bearing on the algebraic properties of $I$. But once we take powers of $I$, the ideal $I^r$ is no longer squarefree, and we quickly leave the realm of Stanley--Reisner theory. 

A widely used method to turn questions about general monomial ideals into questions about squarefree ones is {\it polarization}, but there is no known systematic description of the polarization of powers of monomial ideals.

 In this paper, starting from the edge ideal $I(G)$ of a graph $G$, we give a full description of the polarization $\MP(I(G)^2)$. We then construct the Stanley--Reisner complex of the square-free monomial ideal $\MP(I(G)^2)$, describing each facet in terms of the combinatorial structure of the graph $G$. In doing so, we simultaneously extend the techniques  
of graph theory and introduce the tools of Stanley--Reisner theory in the study of the powers of edge ideals.

 As an application, we address the question: when is the square $I(G)^2$ of an edge ideal $I(G)$  Cohen--Macaulay?  The Cohen-Macaulay property is nearly impossible  to characterize combinatorially,
 but for square-free monomial ideals there is a criterion given by Reisner \cite{Rei} in the language of algebraic topology on simplicial complexes.  Reisner's criterion was dramatically used by Stanley \cite{Sta} for solving affirmatively the upper bound conjecture for the number of faces of simplicial spheres.  But even for an ideal generated by quadratic squarefree monomials, in other words the edge  ideal of a graph, it is in general rather difficult to determine the Cohen--Macaulay property, though for some classes complete characterizations are known (\cite{HHbipartite, HHZchordal, V}).

Our paper is organized as follows. We  recall fundamental concepts on finite graphs and simplicial complexes in \cref{sec:1}. In \cref{sec:2}, we explicitly describe the structure of the Stanley--Reisner complex of the square $I(G)^2$ of the edge ideal of a finite graph $G$ in terms of $G$ (\cref{t:facets}). We then give a criterion for the Stanley--Reisner complex of $I(G)^2$ to be pure (\cref{t:pure}).  This leads to a  necessary condition on $G$ in order for $I(G)^2$ to be Cohen--Macaulay (\cref{t:nice}). As a result, we show that if $G$ is a  cycle, a whiskered graph, a tree, a connected chordal graph, or  a connected Cohen--Macaulay bipartite graph, then $I(G)^2$ is not Cohen--Macaulay, unless  $G$ is a pentagon or $G$ consists of exactly one edge (\cref{c:cycles} and \cref{c:aaaaa}).

Our  results on the Cohen--Macaulayness of $I(G)^2$ also recover earlier works on this topic. The authors in \cite{HT} prove that $I(G)^2$ is Cohen--Macaulay if and only if $G$ is a triangle-free Gorenstein graph, and in \cite{HRT} that $I(G)^2$ is Cohen--Macaulay if and only if $I(G)^2$ satisfies the Serre condition (S$_2$).  Furthermore, in \cite[Section 4]{HRT}, a classification of finite graphs $G$ for which $I(G)^2$ is Cohen--Macaulay with dimension less than $5$ is given.  On the other hand, in \cite{RTY}, examples of Stanley--Reisner ideals whose second powers are Cohen--Macaulay are provided. In \cite{HT}, finite graphs being well-covered together with being a member of the $W_2$ class play an important role.  These notions also appear implicitly in our argument using Reisner's criterion \cite{Rei}.    
We thank the referee for informing us about the references \cite{HT, HRT}, and many helpful comments which improved this paper.

\section{Finite graphs and simplicial complexes}\label{sec:1}

First, we recall fundamental facts on finite graphs.  Let $G$ be a finite graph on the vertex set $V(G)$ with no loop, no multiple edge and no isolated vertex and $E(G)$ the set of edges of $G$.  

A subset $A$ of $V(G)$ is called an {\bf independent set} if $v$ and $w$ belong to $A$ with $v \neq w$, then $\{v,w\} \not\in E(G)$.  Let $\A(G)$ denote the set of maximal independent sets of $G$.  The {\bf independence number} of $G$, denoted by $\alpha(G)$, is the largest possible cardinality of an independent set of $G$.  We say that $G$ is {\bf unmixed} if all maximal independent sets have the same cardinality.

Let $v$ be a vertex of $G$.  The {\bf neighborhood} of $v$ in $G$ is the set $N_G(v)$ of vertices adjacent to $v$.
$$N_G(v)=\{w \in V(G) \st \{v,w\} \in E(G)\}$$
The {\bf closed neighborhood} of $v$ in $G$ is the set $N_G[v]$ of vertices adjacent to $v$ together with $v$ itself:
$$N_G[v]=N_G(v) \cup \{v\}.$$
The closed neighborhood of a subset $W \subseteq V(G)$ is defined as 
$$N_G[W]=\bigcup_{w \in W} N_G[w].$$

If $U \subseteq V(G)$, then the  {\bf induced subgraph} $G|_U$ of $G$ on $U$ is the  subgraph of $G$ whose vertex set is $U$, and whose edge set is all the edges of $G$ with both vertices in $U$:
$$E(G|_U)=\{\{u,v\} \st  \ u,v \in U \qand \{u,v\} \in E(G)\}.$$
The {\bf deletion} of $U$ from $G$ is the induced subgraph $$G \setminus U = G_{V(G) \setminus U}$$ of $G$ on $V(G) \setminus U$, which may include isolated vertices.

A {\bf star} in $G$ is an induced subgraph consisting of edges $\{a^1,b\}, \ldots,\{a^t,b\}$, where the vertex $b$ is the {\bf center} of the star.

A {\bf leaf} of $G$ is a free vertex, that is, a vertex which belongs to exactly one edge of $G$ called a {\bf leaf edge}. 

\begin{example}\label{ex:stars} For the graph $G$ in \cref{fig:stars}~(left), the induced subgraph on $\{x,y,z,w\}$ (center) is not a star as it contains a triangle, while the induced subgraph on $\{x,y,u,w\}$ (right) is a star centered at $x$. 
\begin{figure}[ht!]
\begin{tabular}{ccc}
\begin{tikzpicture}[scale=.8]
\tikzstyle{point}=[inner sep=0pt]
\node[label=left:{\small$z$}] (z) at  (0,0) {};
\node[label=below:{\small$x$}] (x) at (2,0) {};
\node[label=above:{\small$y$}] (y) at (1,1) {};
\node[label=right:{\small$w$}] (w) at (3.5,0) {};
\node[label=above:{\small$u$}] (u) at (2.5,1) {};
\node[label=above:{\small$v$}] (v) at (4,1) {};
\draw (z.center) -- (x.center);
\draw (x.center) -- (y.center);
\draw (x.center) -- (w.center);
\draw (x.center) -- (u.center);
\draw (w.center) -- (v.center);
\draw (z.center) -- (y.center);
\foreach \i in {x,y,z,u,v,w} {\draw[black, fill=black] (\i) circle(.05);}
\end{tikzpicture}
&
\begin{tikzpicture}[scale=.8]
\tikzstyle{point}=[inner sep=0pt]
\node[label=left:{\small$z$}] (z) at  (0,0) {};
\node[label=below:{\small$x$}] (x) at (2,0) {};
\node[label=above:{\small$y$}] (y) at (1,1) {};
\node[label=right:{\small$w$}] (w) at (3.5,0) {};

\draw (z.center) -- (x.center);
\draw (x.center) -- (y.center);
\draw (x.center) -- (w.center);
\draw (z.center) -- (y.center);
\foreach \i in {x,y,z,w} {\draw[black, fill=black] (\i) circle(.05);}
\end{tikzpicture}
&\begin{tikzpicture}[scale=.8]
\tikzstyle{point}=[inner sep=0pt]
\node[label=below:{\small$x$}] (x) at (2,0) {};
\node[label=above:{\small$y$}] (y) at (1,1) {};
\node[label=right:{\small$w$}] (w) at (3.5,0) {};
\node[label=above:{\small$u$}] (u) at (2.5,1) {};

\draw (x.center) -- (y.center);
\draw (x.center) -- (w.center);
\draw (x.center) -- (u.center);
\foreach \i in {x,y,u,w} {\draw[black, fill=black] (\i) circle(.05);}
\end{tikzpicture}
\\
The graph&
The induced subgraph on &
The induced subgraph on \\
$G$& $\{x,y,z,w\}$ & $\{x,y,u,w\}$\\
&(non-star)
&(star centered at $x$)\\
\end{tabular}
\caption{\cref{ex:stars}}\label{fig:stars}
\end{figure}

Both vertices $u$ and $v$ are leaves of $G$. If $W=\{u,z\}$, then the closed neighborhood of $W$ is $N_G[W]=\{x,y,z,u\}$. Finally, we have
$$\A(G)=\big \{\{x,v\}, \{y,u,v\}, \{y,u,w\},\{z,u,v\}, \{z,u,w\} \big \}.$$

\end{example}
\begin{definition}[{\bf edge ideal}]
Let $G$ be a finite graph on the vertex set $V=V(G)$.  Let $S=K[v \st v \in V]$ denote the polynomial ring in $|V|$ variables over a field $K$.  The {\bf edge ideal} of $G$ is the monomial ideal $I(G)$ of $S$ generated by those monomials $vw$ with $\{v,w\} \in E(G)$.  
\end{definition}

Second, we recall fundamental facts on simplicial complexes.  Let $V$ be a finite set.  A {\bf simplicial complex} on $V$ is a collection $\Gamma$ of subsets of $V$ for which (i) $\{v\} \in \Gamma$ for each $v \in V$ and (ii) if $\sigma \in \Gamma$ and $\tau \subset \sigma$, then $\tau \in \Gamma$.  Each element $\sigma \in \Gamma$ is called a {\bf face} of $\Gamma$.  A face that is maximal with respect to inclusion is called a {\bf facet} of $\Gamma$.  Let $d$ denote the largest possible cardinality of a face of $\Delta$.  The {\bf dimension} of $\Gamma$ is $\dim \Gamma = d-1$.  A simplicial complex is {\bf pure} if all facets have the same cardinality.  

\begin{definition}[{\bf Stanley--Reisner complex}]
Let $V$ be a finite set and $S=K[v \st v \in V]$ the polynomial ring in $|V|$ variables over a field $K$.  Let $I$ be an ideal of $S$ generated by squarefree monomials of degree $\geq 2$.  The {\bf Stanley--Reisner complex} of $I$ is a simplicial complex $\Gamma_{I}$ on $V$ whose faces are those $\sigma \subset V$ with 
$$\bm_\sigma = \prod_{v \in \sigma} v \not\in I.$$
It follows easily that, given a simplicial complex $\Gamma$ on $V$, there is an ideal $I$ of $S$ generated by squarefree monomials for which $\Gamma = \Gamma_I$.
\end{definition}

\begin{definition}[{\bf Cohen--Macaulay complexes}]
Let $I$ be an ideal of $S$ generated by squarefree monomials and $\Gamma = \Gamma_I$ its Stanley--Reisner complex.  
We say that $\Gamma$ is {\bf Cohen--Macaulay} over $K$ if $I$ is Cohen--Macaulay, i.e. $S/I$ is Cohen--Macaulay. 
\end{definition}

Let $\Gamma$ be a simplicial complex and $\sigma$ a face of $\Gamma$.  Then the {\bf link} of $\sigma$ in $\Gamma$ is the subcomplex
\[
\lk_\Gamma(\sigma) = \{ \tau \in \Gamma \st \sigma \cap \tau = \emptyset, \, \sigma \cup \tau \in \Gamma \}. 
\]
In particular,  $\lk_\Gamma(\emptyset) = \Gamma$.

\begin{theorem}[{\bf Reisner's criterion~\cite{Rei}}]
\label{t:Reisner}
A simplicial complex is Cohen--Macaulay over $K$ if and only if for each face $\sigma$ of $\Delta$ (including $\sigma = \emptyset$), one has
\[
{\tilde H}_i(\lk_\Gamma(\sigma);K) = 0, \qforall i \neq \dim \lk_\Gamma(\sigma).
\]
\end{theorem}

The following statement is an immediate consequence of Reisner's criterion.

\begin{corollary}\label{c:Reisner} Every Cohen--Macaulay complex is connected and pure, and  every link of a Cohen--Macaulay complex is Cohen--Macaulay.
\end{corollary}

A simple way of extending Reisner's criterion to all monomial ideals from the class of squarefree ones is a technique called \say{polarization}, which transforms a monomial ideal into a squarefree one via a regular sequence. In particular, a monomial ideal is Cohen--Macaulay if and only if its polarization is Cohen--Macaulay (see \cite[Corollary 1.6.3]{HHgtm260}).

\begin{definition}[{\bf Polarization}]\label{d:polarization}
 Let $S=K[x_1,\ldots,x_n]$ be a polynomial ring, let $i \in [n]$ and let $a_i$ be a non negative integer. We define the {\bf polarization} $\MP(x_i^{a_i})$ of $x_i^{a_i}$ to be $1$ if $a_i=0$, and otherwise the squarefree monomial in an extended polynomial ring
 $$\MP(x_i^{a_i})=x_{i,1}x_{i,2}\cdots x_{i,a_i} \in 
 K[x_1,\ldots,\hat{x_i},\ldots,x_n,x_{i,1},\ldots,x_{i,a_i}].$$
 We define the polarization of a monomial $m=x_1^{a_1}\cdots x_n^{a_n}$ in the polynomial ring $S$ to be the squarefree monomial 
$\MP(m)=\MP(x_1^{a_1})\cdots \MP(x_n^{a_n})$, and the  polarization of an $I$ ideal minimally generated by monomials $m_1,\ldots,m_q$ of $S$ to be 
the squarefree monomial ideal $$\MP(I)= \big ( \MP(m_1),\ldots,\MP(m_q) \big )$$
in a polynomial ring extension of $S$.
\end{definition}

\section{Squares of edge ideals and Stanley--Reisner complexes}\label{sec:2}
Our goal in this paper is to investigate which graphs $G$ with edge ideal $I(G)$ have the property that $I(G)^2$ is Cohen--Macaulay. Our main tool is the application 
of Reisner's criterion to the Stanley--Reisner complex  $\Gamma^2_{G}$ of the polarization of $I(G)^2$.   Following our notation for polarization,  when $G$ is a  graph on vertex set $V$, then the vertex set of $\Gamma^2_G$ is 
$$
V(\Gamma^2_G)=V_{(1)} \cup V_{(2)} \qwhere 
V_{(1)}=\{v_1 \st v \in V\} \qand
V_{(2)}=\{v_2 \st v \in V\}.
$$

We begin with an example.

\begin{example}[{\bf $\Gamma^2_{P_3}$ is not Cohen--Macaulay}]\label{e:P3}  
Let $P_3$ be the path of length $3$ with the vertices $x,y,z,w$ and with  edges $\{z,x\}, \{x,y\}, \{y,w\}$.  

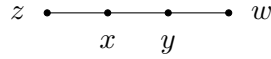
$$\begin{tikzpicture}[scale=.8]
\tikzstyle{point}=[inner sep=0pt]
\node[label=left:{\small$z$}] (z) at (0,0) {};
\node[label=below:{\small$x$}] (x) at (1,0) {};
\node[label=below:{\small$y$}] (y) at (2,0) {};
\node[label=right:{\small$w$}] (w) at (3,0) {};
\draw (z.center) -- (x.center);
\draw (x.center) -- (y.center);
\draw (y.center) -- (w.center);
\draw[black, fill=black] (x) circle(.05);
\draw[black, fill=black] (y) circle(.05);
\draw[black, fill=black] (z) circle(.05);
\draw[black, fill=black] (w) circle(.05);
\end{tikzpicture}
$$
The edge ideal of $P_3$ above is $I(P_3)=(xz,xy,yw)$, and hence the polarization of the square $I(P_3)^2$ is the following squarefree monomial ideal
$$
\MP(I(P_3)^2)= 
(x_1x_2z_1z_2,
x_1x_2z_1y_1,
x_1z_1y_1w_1,
x_1x_2y_1y_2,
x_1y_1y_2w_1,
y_1y_2w_1w_2).
$$
The facets of the Stanley--Reisner complex $\Gamma^2_{P_3}$ of  $\MP(I(P_3)^2)$ are  
\begin{eqnarray*}
\{x_1,w_1,x_2,y_2,z_2,w_2\},
\{x_1,y_1,z_1,y_2,z_2,w_2\},
\{x_1,y_1,w_1,x_2,z_2,w_2\},\\
\{y_1,z_1,x_2,y_2,z_2,w_2\},
\{x_1,z_1,w_1,y_2,z_2,w_2\},
\{y_1,w_1,z_1,x_2,z_2,w_2\},\\
\{z_1,w_1,x_2,y_2,z_2,w_2\},
\{x_1,z_1,w_1,x_2,y_2,w_2\},
\{y_1,w_1,z_1,x_2,y_2,z_2\}.
\end{eqnarray*}

Let $\sigma = \{x_1,y_1,z_2,w_2\} \in \Gamma_{P_3}^2$.  The facets of $\lk_{\Gamma^2_{P_3}}(\sigma)$ are $$\{z_1, y_2\}, \, \, \, \{w_1, x_2\}.$$ Since $\lk_{\Gamma^2_{P_3}}(\sigma)$ is disconnected, Reisner's criterion implies that it is not acyclic, and hence $\lk_{\Gamma^2_{P_3}}(\sigma)$ cannot be Cohen--Macaulay.  It then follows from \cref{c:Reisner} that $\Gamma^2_{P_3}$ itself cannot Cohen--Macaulay.
\end{example}

We now give a concrete description of the facets of $\Gamma^2_G$ for any graph $G$ in terms of the shape of $G$ itself. We encourage the reader to compare the facets listed below with those computed in  \cref{e:P3}.

\begin{theorem}[{\bf The Stanley--Reisner complex of the polarization of $I(G)^2$}]\label{t:facets} Let $G$ be a finite graph on vertex set $V$  and let  $\Gamma^2_G$ be the Stanley--Reisner complex of the polarization of $I(G)^2$. Then the facets of $\Gamma^2_G$ are the maximal subsets of $V_{(1)} \cup V_{(2)}$ of the form
$$W_{(1)} \cup A_{(1)} \cup Z_{(2)}$$ where $W$ is a maximal subset of $V$ in one of the following forms:
\begin{itemize}

\item {\bf independent type:}
$W=\emptyset$,
 $$A \in \A(G)  \qand  
 Z=V;$$

\item {\bf leaf type:} 
$W=\{a,b\}$ where $G_W$ is a leaf edge of $G$ with free vertex $a$, 
$$A \in \A(G\setminus N_G[b]), \qand 
Z=V \setminus\{a\};$$

\item {\bf star type:}
$W=\{a^1,\ldots,a^t,b\}$ where the induced subgraph $G_W$ is a star centered at $b$,  
$$A \in \A \big (G \setminus N_G[W]), \qand  
Z=V \setminus\{b\};$$

\item {\bf triangle type:} 
 $W=\{a,b,c\}$ where $G_W$ is a triangle, 
$$A \in \A \big (G \setminus N_G[W] \big ), \qand  
Z=V \setminus\{a,b,c\}.$$
\end{itemize}

In particular, if $\{a,b\}$ is an edge of $G$, then $\{a, b\}_{(1)}$ is an edge of $\Gamma^2_G$. 
\end{theorem}

Before presenting a proof of \cref{t:facets}, we provide an example.

\begin{example}[{\bf The case of a triangle}]\label{ex:triangle} 
Below we demonstrate Macaulay2~\cite{M2} code for finding the facets of $\Gamma^2_G$ when $G$ is a triangle with vertices $a,b,c$. Observe that $G$ is unmixed, but  $\Gamma^2_G$ is not pure. It turns out that the triangle is the main obstruction to $\Gamma^2_G$  being pure if $G$ is an unmixed graph. This will be explored in \cref{t:pure} below.
\begin{verbatim}
I=monomialIdeal(a*b,b*c,c*a);
J=I^2;
H= polarize J;
D=simplicialComplex H
 simplicialComplex
| a_1b_1c_1      -- (triangle)
  a_2b_2c_1c_2   -- (independent set c)
  a_1b_2c_1c_2   -- (star ac centered at a)
  a_2b_1c_1c_2   -- (star bc centered at b)
  a_2b_1b_2c_2   -- (independent set b)
  a_1b_1b_2c_2   -- (star ab centered at a)
  a_1a_2b_2c_2   -- (independent set a)
  a_1a_2b_1c_2   -- (star ab centered at b)
  a_2b_1b_2c_1   -- (star bc centered at c)
  a_1a_2b_2c_1 | -- (star ac centered at c)
\end{verbatim}
\end{example}

\begin{proof}[Proof of \cref{t:facets}] 

The ideal $J
=\MP(I(G)^2)$ has generators of the form, for distinct vertices $a,b,c,d$: 
\begin{equation}
\label{e:possible}
\begin{tabular}{ll}
$a_1a_2 b_1b_2$ \qwhere $\{a,b\}$ &is an edge of $G$\\
$a_1b_1b_2c_1$ \qwhere $\{a,b\}, \{b,c\}$ &are edges of $G$ (a star)\\
$a_1b_1c_1d_1$ \qwhere $\{a,b\},\{c,d\}$ &are disjoint edges of $G$ (a matching)
\end{tabular}
\end{equation}

So every face of $\Gamma^2_G$ must avoid these three types of relations. In particular, the facets are the maximal such sets. Therefore, for subsets $V'$ and $V''$ of $V$,  $\sigma=V'_{(1)} \cup V''_{(2)}$ is a face of $\Gamma^2_G$ if and only if for distinct vertices $a,b,c,d \in V'$,

$$\begin{tabular}{ll}
$\{a,b\} \in E(G) \implies a \notin V'' \qor b \notin V''$ &(otherwise $\{a_1,a_2,b_1,b_2\} \in \Gamma^2_G$)\\
$\{a,b\}, \{b,c\} \in E(G)\implies b \notin V''$  &(otherwise $\{a_1,b_1,b_2, c_1\} \in \Gamma^2_G$)\\
$\{a,b\}, \{c,d\} \notin E(G)$ &(otherwise $\{a_1,b_1,c_1,d_1\} \in \Gamma^2_G$).
\end{tabular}
$$

The most basic type of facet, then, comes from maximal independent sets of vertices where there is no edge between them. In this case, all of $V_{(2)}$ can be part of the facet. 
These are the facets of \say{independent type}.

Next we consider facets of \say{leaf type}, that is,  a facet which contains $a_1,b_1$ and $a$ is a free vertex of $G$ connected to $b$. In this case, we consider the set 
$$\sigma=\{a,b\}_{(1)} \cup A_{(1)} \cup (V\setminus\{a\})_{(2)} \qwhere A \in \A (G \setminus N_G[b]).$$

First we claim that $\sigma$ is a face of $\Gamma^2_G$. This is because the only edge of $G$ with vertices in $\{a,b\} \cup A$ is the edge $\{a,b\}$, and so  for $\bm_\sigma$ to be divisible by a generator of $J$, $\sigma$ needs to contain $a_2$. We next consider the maximality of $\sigma$.  If we add $c_1$ to $\sigma$ for some vertex  $c$ of $G$ with $c_1 \notin \sigma$, then  either $c \in N_G[b]$ or $c$ is a vertex of $G \setminus N_G[b]$. This means that either  $\{b, c\}$ is an edge of $G$ or  $\{c, d\}$ is an edge of $G$ for some $d \in A$ (which necessarily implies that $d \notin N_G[b]$). In the former scenario the generator $a_1b_1b_2c_1$  and in the latter scenario the generator  $a_1b_1c_1d_1$ of $J$ will divide $\bm_\sigma$. If we add $a_2$ then $a_1a_2b_1b_2 \mid \bm_\sigma$. So $\sigma$ is a facet of $\Gamma^2_G$.

Next, we consider a star centered at $b$, with edges $\{a^1,b\}, \ldots,\{a^t,b\}$, and vertex set $W=\{a^1,\ldots,a^t,b\}$ so that  
$$\sigma=\{a^1,\ldots,a^t,b\}_{(1)} \cup A_{(1)} \cup (V\setminus\{b\})_{(2)} \qwhere A \in \A(G \setminus  N_G[W]).$$
To see that $\sigma$ is a face, we observe that since $\sigma$ does not contain disjoint edges, and the only  edges with vertices  in  $W \cup A$ are $\{a^1,b\}, \ldots, \{a^t,b\}$, for $\bm_\sigma$ to be divisible by a generator in \eqref{e:possible} it needs to contain the vertex  $b_2$, but  $b_2 \notin \sigma$.
Moreover $\sigma$ is maximal because adding $b_2$ forces  $\bm_\sigma \in J$ as discussed earlier, and adding a new vertex $c_1$ will create 
a new edge of either of the following forms:
\begin{itemize}
    \item $\{a^i,c\}$ for some $i$ which creates a \say{star} generator of $J$
centered at $a^i$ in \eqref{e:possible} forcing $\bm_\sigma \in J$, or
\item $\{b,c\}$, in which case since we are assuming that $W$ is maximal with the property that $G_W$ is a  star, we must have an edge $\{a^i,c\}$ is an edge of $G$ for some $i \in [t]$,  which was covered in the previous case, or 
\item $\{c,d\}$ which creates a \say{matching} generator of $J$ in \cref{e:possible}, again forcing $\bm_\sigma \in J$. 
     \end{itemize}
Therefore $\sigma$ is a facet of \say{star type}.

Finally, we consider a triangle $\{a,b\},\{a,c\},\{b,c\}$ in $G$ with $W=\{a,b,c\}$, and we claim
$$\sigma=\{a_1,b_1,c_1\} \cup A_{(1)} \cup (V\setminus\{a,b,c\})_{(2)} \qwhere A \in \A(G \setminus  N_G[W]).$$
According to \eqref{e:possible}, $\sigma \in \Gamma^2_G$ if and only if  $a_2, b_2, c_2 \notin \sigma$, so $\sigma \in \Gamma^2_G$. For the same reason $\sigma$ is maximal, as if we add any $d_1$ to $\sigma$, we will have a new edge $\{d_1,a_1\}$ or$\{d_1,b_1\}$, $\{d_1,c_1\}$ or $\{d_1,e_1\}$ for some $e \in A$, in which case we will have either $a_1b_1c_1d_1 \mid \bm_\sigma$, or $a_1b_1d_1e_1 \mid \bm_\sigma$ forcing $\sigma \notin \Gamma^2_G$. Hence  $\sigma$ is a facet of $\Gamma^2_G$ of \say{triangle type}. 
\end {proof}

\begin{lemma}\label{l:leaf} Let $G$ be a finite graph with a vertex $b$, and suppose $A$ is a (maximal) independent set of 
      $G\setminus N_G[b]$.  Then $A\cup\{b\}$ is a (maximal) independent set of $G$.
\end{lemma}

\begin{proof} Suppose $A$ is a (maximal) independent set of 
      $G\setminus N_G[b]$. Then $b$ is connected to no vertex in $A$ by construction, and hence $A \cup \{b\}$ is an independent set in $G$. Now suppose $A$ was picked to be maximal. Then we claim $A\cup\{b\}$ is also maximal: the maximality of $A$ implies that any other vertex of $G\setminus N_G[b]$ is connected to some vertex in $A$ by an edge. Moreover, if we pick a vertex $v$ in $N_G[b]$, then $v$ is either $b$ itself or is connected to $b$ via an edge. Since $A\cup\{b\}$ cannot be enlarged without losing its independence, it must be maximal.
        \end{proof}

\begin{theorem}[{\bf Purity of $\Gamma^2_G$}]\label{t:pure}
Let $G$ be a finite graph on $n \geq 1$ vertices. Then $$\dim(\Gamma^2_G) \geq n.$$ Moreover,  $\Gamma^2_G$ is pure (of dimension $n +\alpha(G) -1$) if and only if $G$ is unmixed and contains no triangle. 
\end{theorem}

\begin{proof} By \cref{t:facets}, we know the facets of $G$ with vertex set $V$ are of the form 
$$\sigma=W_{(1)} \cup A_{(1)} \cup Z_{(2)}$$ with conditions on the sets $W$, $A$ and $Z$. In each case we calculate the size of the corresponding facet.
Assuming $|V|=n$, we must have $W$  a maximal subset of $V$ such that 

\begin{itemize}

\item (independent facets) $W=\emptyset$,  
      $A \in \A(G)$, and  
      $Z=V$, in which case 
      $$|\sigma|=n+|A| \leq n+\alpha(G);$$

\item (leaf facets) $W=\{a,b\}$ where $\{a,b\}$ is a leaf edge of $G$ with free vertex $a$, 
      $A \in \A(G\setminus N_G[b])$, and 
      $Z=V \setminus\{a\}$, then 
      $$|\sigma|=n+1+|A|  \leq n+1 +\alpha(G\setminus N_G[b]);$$

\item (star facets) $W=\{a^1,\ldots,a^t,b\}$ where $G_W$ is a star centered at $b$,  
      $A \in \A \big (G \setminus N_G[W] \big )$ and  
      $Z=V \setminus\{b\}$, in which case 
      $$|\sigma|=n+t+|A| \leq n+t+\alpha(G\setminus  N_G[W]);$$

\item (triangle facets) $W=\{a,b,c\}$ where $G_W$ is a triangle, 
      $A \in \A \big (G \setminus N_G[W] \big )$, and  
      $Z=V \setminus\{a,b,c\}$, then 
      $$|\sigma|=n+|A| \leq n+\alpha(G\setminus N_G[W]).$$

\end{itemize}

The simplicial complex $\Gamma^2_G$ always has an independent type facet $\sigma$ which corresponds to a  maximal independent set $A$ of $G$. Since $G$ has one or more vertices, and a set of one vertex  is an independent set, it follows that $$\dim(\Gamma^2_G) \geq \dim(\sigma) = |\sigma| -1 \geq n+1-1=n.$$ 

Now suppose $\Gamma^2_G$ is pure. Then all facets of independent type have the same cardinality equal to $n + \alpha(G)$ and in particular, all maximal independent  
sets of $G$ have cardinality equal to $\alpha(G)$, implying that $G$ is unmixed, and also $\dim(\Gamma^2_G)=n+\alpha(G)-1$.

If $G$ contains a triangle with vertices $a,b,c$, then it will have a facet of triangle type, which has size equal to  $n+|A|$  where $A$ is a maximal independent set of $G\setminus N_G[W]$. This implies that 
\begin{equation}\label{e:alphas}
\alpha(G)=\alpha \big (G\setminus N_G[W] \big ) \qwhere W=\{a,b,c\} \mbox{ is a triangle}.
\end{equation}
However observe that
if $A$ is an independent set of $G\setminus N_G[W]$, then $A \cup\{a\}$ is an independent set of $G$, which implies that 
$$\alpha(G) \geq \alpha \big (G\setminus N_G[W] \big ) +1,$$ contradicting \eqref{e:alphas}. Therefore, we have proved that if $\Gamma^2_G$ is pure, then $G$ is unmixed and triangle-free.

Conversely, assume that $G$ is unmixed and triangle-free. Then all facets of independent type will have size equal to $n+\alpha(G)$. By \cref{l:leaf}, all facets of leaf type will have size
$n+1+\alpha(G)-1=n+\alpha(G)$. Since $G$ is triangle free, there are no facets of triangle type in $\Gamma^2_G$, so we consider a facet of star type  
$$\sigma_W=W_{(1)} \cup A_{(1)} \cup (V\setminus\{b\})_{(2)}$$ where 
$W=\{a^1,\ldots,a^t,b\}$ is maximal, $G_W$ is a star centered at $b$, and $A$ is a maximal independent set of 
      $G'=G\setminus N_G[W]$. 

We claim that $A\cup\{a^1,\ldots,a^t\}$ is a maximal independent set of $G$. Since $a^1,\ldots,a^t$ are not connected to any vertex in $A$ by construction,  $A \cup \{a^1,\ldots,a^t\}$ is an independent set in $G$. The maximality of $A$ implies that any other vertex of $G'$ is connected to some vertex in $A$ by an edge. 
Suppose there is a vertex $v$ such that  
$$ v \in N_G[W] \qand A'=A \cup\{a^1,\ldots,a^t\} \cup\{v\} \in \A(G).$$ 
Since $v$ is independent of $a^1,\ldots, a^t$, we have moreover  $$v \notin N_G[a^1] \cup \cdots\cup N_G[a^t].$$
Therefore we must have 
\begin{itemize}
\item $v \in N_G(b)\setminus\{a^1,\ldots,a^t\}$, 
\item since $G$ is triangle-free, $G_{W'}$ is a star where $W'=\{v,a^1,\ldots,a^t,b\}$,
\item no vertex in $A$ is adjacent to $v$, and so all vertices in $A$ remain vertices of the subgraph $G''=G\setminus N_G[W']$, and $A$ is therefore still an independent set of $G''$.
\end{itemize}
We have hence  created a star type face
$$\sigma_{W'}=W'_{(1)} \cup A_{(1)} \cup (V\setminus\{b\})_{(2)} \in \Gamma^2_G \qand 
\sigma_W \subsetneq \sigma_{W'}
$$
contradicting $\sigma_W$ being a facet. Therefore $A\cup\{a^1,\ldots,a^t\}$ must 
be a maximal independent set of $G$, implying that $$\alpha(G')=\alpha(G)-t.$$ In particular all facets of star type also have size $n+\alpha(G)$. This ends our argument.
\end{proof}

 \section{Classifications}\label{sec:3}

Which finite graphs have Cohen--Macaulay squares?  In the present section, we classify finite graphs with Cohen--Macaulay squares in particular classes of finite graphs.  One of the immediate corollaries of \cref{t:pure} is the \cref{c:cycles}. We first record a well-known fact.

\begin{lemma}\label{l:unmixed}
The cycle  $C_t$ on $t$ vertices is unmixed if and only if $t \in \{3,4,5,7\}$. 
\end{lemma}

\begin{proof} The cases $t \leq 7$ can be checked by hand or using Macaulay~2. Assume $t \geq 8$. 
Using the division algorithm, we can write 
$$
t=2a +r=3b+s \qwhere 
a,b\geq 2, \quad  
r \leq 1 \qand
s \leq 2.
$$
Then $C_t$ always has a  maximal independent set of size $a$
\begin{equation}\label{eq:A}
\{x_2,x_4,\ldots,x_{2a}\}.
\end{equation}

Now consider two scenarios.
\begin{enumerate}
\item If $s=0$, then $C_t$ has a maximal independent set of size $b$
$$
\{x_3,x_6,\ldots,x_{3b}\}.
$$
If this set has the same cardinality as the one in \eqref{eq:A}, then $a=b$ and we have $t=2b+r=3b$, implying that $b=r\leq 1$, which contradicts the fact that $b\geq 2$. 

\item If $s\in\{1,2\}$, then $C_t$ has a maximal independent set of size $b+1$
$$
\{x_3,x_6, \ldots,x_{3b}\} \cup \{x_1\}.
$$
If this set has the same cardinality as the one in \eqref{eq:A}, then $a=b+1$ and we have $t=2(b+1)+r=3b+s$, implying that $b=r+2-s$. As $r-s\leq 0$ and $b \geq 2$, the only possibility is $r=s=1$ and $b=2$. Then
$t=3\cdot 2+1=7 <8$, a contradiction to our assumption that $t\geq 8$.
\end{enumerate}
We conclude that when $t\geq 8$, the cycle $C_t$ is not unmixed.
 \end{proof}

\begin{theorem}[{\bf When $\Gamma^2_{C_t}$ is Cohen--Macaulay}]\label{c:cycles}
If $C_t$ denotes a cycle on $t$ vertices, then $I(C_t)^2$ is Cohen--Macaulay if and only if $t=5$.  
\end{theorem}

\begin{proof}  By virtue of  \cref{t:pure,l:unmixed} $I(C_t)^2$ cannot be Cohen--Macaulay unless $t \in \{3,4,5,7\}$.   On the other hand, computations by Macaulay2 demonstrate that $I(C_t)^2$ is not Cohen--Macaulay for $t \in \{3,4,7\}$ and that $I(C_5)^2$ is Cohen--Macaulay.
\end{proof}

We are now in the position to give a powerful technique to show, in \cref{t:nice}, that plenty of finite graphs cannot have Cohen--Macaulay squares. \cref{l:AAAAA} provides a crucial ingredient.

\begin{lemma}\label{l:AAAAA}
Let $G$ be a bipartite graph on the vertex set 
$$V(G) = \{x_0, y_1, \ldots, y_s\} \cup 
\{y_0, x_1, \ldots, x_t\}$$ 
with $s \geq 1$ and $t \geq 1$ and suppose that 
$$N_G(x_0) = \{y_0, x_1, \ldots, x_t\}, \qquad 
N_G(y_0) = \{x_0, y_1, \ldots, y_s\}.
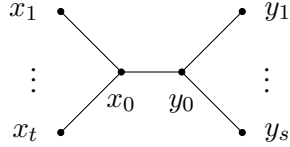
$$  
$$\begin{tikzpicture}[scale=.8]
\tikzstyle{point}=[inner sep=0pt]
\node[label=left:{\small$x_1$}] (x1) at (0,1) {};
\node[label=left:{\small$\vdots$}] (x2) at (0,0) {};
\node[label=left:{\small$x_t$}] (xt) at (0,-1) {};
\node[label=below:{\small$x_0$}] (x) at (1,0) {};
\node[label=below:{\small$y_0$}] (y) at (2,0) {};
\node[label=right:{\small$\vdots$}] (y2) at (3,0) {};
\node[label=right:{\small$y_1$}] (y1) at (3,1) {};
\node[label=right:{\small$y_s$}] (ys) at (3,-1) {};

\foreach \i in {x1,xt,y} {\draw (x.center) -- (\i.center);}
\foreach \i in {y1,ys} {\draw (y.center) -- (\i.center);}
\foreach \i in {x,x1,xt,y,y1,ys} {\draw[black, fill=black] (\i) circle(.05);}
\end{tikzpicture}
$$
Then 
$\Gamma^2_{G}$ is not Cohen--Macaulay. 
\end{lemma}
\begin{proof}
Using \cref{t:facets}, let $$\sigma = (N_G[x_0])_{(1)} \cup (V(G) \setminus \{x_0\})_{(2)} \qand
\tau = (N_G[y_0])_{(1)} \cup (V(G) \setminus \{y_0\})_{(2)}$$ be facets of $\Gamma^2_{G}$.  
Let $\pi = \{x_0,y_0\}_{(1)} \cup (V(G) \setminus \{x_0,y_0\})_{(2)}$ be a face of $\Gamma^2_{G}$.  
Since the facets of $\Gamma^2_{G}$ containing $\{x_0,y_0\}_{(1)}$ are $\sigma$ and $\tau$, it follows that the facets of $\lk_{\Gamma^2_{G}}(\pi)$ are 
$$\{y_{1,1}, \ldots, y_{s,1}\} \cup \{x_{0,2}\} \qand  
\{x_{1,1}, \ldots, x_{t,1}\} \cup \{y_{0,2}\}.$$
Since $\lk_{\Gamma^2_{G}}(\pi)$ is disconnected, ${\tilde H}_i(\lk_{\Gamma^2_{G}}(\pi);K) \neq 0$, so $\Gamma^2_{G}$ fails Reisner's criterion and cannot be Cohen--Macaulay.  
\end{proof}

\begin{theorem}[{\bf A necessary condition for $I(G)^2$ to be Cohen--Macaulay}]
\label{t:nice}
If $G$ is a finite connected graph which contains as an induced subgraph a path of length $3$ with two vertices which are leaves of $G$, then $I(G)^2$ is not Cohen--Macaulay.
\end{theorem}

\begin{proof} Let $V$ denote the vertex set of $G$.
First of all, since the purity of $\Gamma^2_G$ is a necessary condition for the Cohen--Macaulayness of $I(G)^2$, by \cref{t:pure} one can assume that $G$ is triangle-free and unmixed.  Let $P_3$ be the path of length $3$, which is an induced subgraph of $G$, with the vertices $x,y,z,w$ and edges $zx,xy,yw$, where $z$ and $w$ are leaves of $G$.  
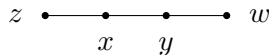
$$\begin{tikzpicture}[scale=.8]
\tikzstyle{point}=[inner sep=0pt]
\node[label=left:{\small$z$}] (z) at (0,0) {};
\node[label=below:{\small$x$}] (x) at (1,0) {};
\node[label=below:{\small$y$}] (y) at (2,0) {};
\node[label=right:{\small$w$}] (w) at (3,0) {};
\draw (z.center) -- (x.center);
\draw (x.center) -- (y.center);
\draw (y.center) -- (w.center);
\draw[black, fill=black] (x) circle(.05);
\draw[black, fill=black] (y) circle(.05);
\draw[black, fill=black] (z) circle(.05);
\draw[black, fill=black] (w) circle(.05);
\end{tikzpicture}
$$
Let 
\begin{equation}\label{e:U}
U = N_G[x] \cup N_G[y], \quad
A \in \A(G\setminus U), \quad  
U_A =\{v \in U \st N_G(v) \cap A = \emptyset\} \subseteq U.
\end{equation}
Observe that 
\begin{equation}\label{e:U0}
x,y,z,w \in  U_A, \quad 
V \setminus U_A= N_G[A], \quad
A \in \A(G \setminus U_A),
\end{equation}
and since $G$ is triangle-free, $G_{U_A}$ is a graph of the form described in \cref{l:AAAAA}.
With $A$ as in \eqref{e:U}, and using \cref{t:facets}, let 
$$\sigma = A_{(1)} \cup (V \setminus U_A)_{(2)} \in \Gamma^2_G.$$
Our goal is to show that $ {\lk}_{\Gamma^2_G}(\sigma)$ is not Cohen--Macaulay, which will then, by~\cref{c:Reisner}, imply that $\Gamma^2_G$ is not Cohen--Macaulay.

\smallskip

\noindent {\bf Claim~1:} {\it Let   $U' \subset U_A$ and $B \in \A(G_{U_{A}\setminus U'})$. Then $A \cup B \in \A(G\setminus U')$.}    

\smallskip

\noindent{\it Proof of Claim~1.}
First note that since $U' \subseteq U_A \subseteq U$, the graphs $G\setminus U$ and $G_{U_{A}\setminus U'}$  are disjoint subgraphs of $G\setminus U'$.
Suppose $A \in \A(G\setminus U)$ and $B \in \A(G_{U_{A}\setminus U'})$.  Let $v \in A$ and $v' \in B$.  Then $v' \in U_{A}$ and by \eqref{e:U} $\{v,v'\}$ cannot be an edge of $G$.  Hence $A \cup B$ is an independent set of $G\setminus U'$.  

We show that $A \cup B$ is a maximal independent set of $G\setminus U'$.  Let 
$v \in V\setminus \big ((A \cup B) \cup U' \big )$.  If $v \not\in U_{A}$, then by \eqref{e:U0} $A \cup \{v\}$ cannot be an independent set.  If $v \in U_{A}\setminus U'$, then $B \cup \{v\}$ cannot be an independent set.  Thus $(A \cup B)\cup \{v\}$ cannot be an independent set of $G\setminus U'$.  Hence $A \cup B \in \A(G\setminus U')$, as required.

\hfill {(\it End of Claim~1.)}

\smallskip

\noindent {\bf Claim~2:} ${\lk}_{\Gamma^2_G}(\sigma) = \Gamma^2_{G_{U_A}}$.  

\smallskip

\noindent{\it Proof of Claim~2.}

\noindent ($\subseteq$) We first show that for $v \in V$ and $i \in \{1,2\}$, if $v_i \in \lk_{\Gamma^2_G}(\sigma)$ then $v \in U_A$.
\begin{itemize}
    \item[$i=1$: ] In this case  $v_1\notin \sigma$ and $\sigma \cup \{v_1\} \in \Gamma^2_G$. If $v \notin U_A$, $v_2 \in \sigma$ and there is $v' \in A$ for which $\{v,v'\} \in E(G)$ by \eqref{e:U0}, and so $vv'\in I(G)$, so $v_1v_2v'_1v'
_2 \in \MP(I(G)^2)$, contradicting  $$\{v_1,v_2,v'_1, v'_2\} \subseteq \sigma \cup \{v_1\} \in \Gamma^2_G.$$

  \item[$i=2$: ] In this case $v_2 \notin \sigma$, and so  $v \in U_A$ by the construction of $\sigma$.
\end{itemize}

In other words, if $\sigma \cup \tau$ is a face of $\Gamma^2_G$ with $\sigma \cap \tau = \emptyset$, then $\tau \subset (U_A)_{(1)} \cup (U_A)_{(2)}$ and thus $\tau \in \Gamma^2_{G_{U_A}}$.  It then follows that $ {\lk}_{\Gamma^2_G}(\sigma) \subseteq \Gamma^2_{G_{U_A}}$.  

\noindent ($\supseteq$) Since 
$$\sigma 
\subseteq (V\setminus U)_{(1)} \cup (V\setminus U_A)_{(2)} 
\subseteq  (V\setminus U_A)_{(1)} \cup (V\setminus U_A)_{(2)}$$ and the vertex set of
$\Gamma^2_{G_{U_A}}$ is $(U_A)_{(1)} \cup (U_A)_{(2)}$, we can see that 
$\sigma \cap \Gamma^2_{G_{U_A}}=\emptyset$. Moreover, as 
$G$ is triangle-free, the facets of $\Gamma^2_{G_{U_A}}$ are of independent, star and leaf types.

 In order to prove that 
${\lk}_{\Gamma^2_G}(\sigma) \supseteq \Gamma^2_{G_{U_A}}$, we recall the facets $\tau$ of $\Gamma^2_{G_{U_A}}$ via \cref{t:facets}.
\begin{enumerate}
 \item{\it independent type:} $\tau=C_{(1)} \cup (U_A)_{(2)}$ where $C \in \A(G_{U_A}).$ 
By Claim~1, $A\cup C \in \A(G)$, and so 
$$ 
\sigma \cup \tau 
= A_{(1)} \cup (V \setminus U_A)_{(2)}  \cup C_{(1)} \cup ({U_A})_{(2)} 
=(A\cup C)_{(1)} \cup V_{(2)}
$$
is a facet of $\Gamma^2_G$ of independent type.
 
 \smallskip
 
 \item{\it leaf type:} $\tau=\{a,b\}_{(1)} \cup C_{(1)} \cup (U_A\setminus\{a\})_{(2)}$ where 
 $\{a,b\}$ is a leaf  of $G_{U_A}$ with free vertex $a$,  and 
  $$C \in \A\big ( G_{U_A} \setminus N_{G_{U_A}}[b] \big ).$$   
By construction, $a$ is also a leaf of $G$, and  by Claim~1, $A\cup C \in \A(G \setminus N_{G_{U_A}}[b])$, and so 
$$ 
\begin{array}{ll}
\sigma \cup \tau&  
= A_{(1)} \cup (V \setminus U_A)_{(2)}  \cup \{a,b\}_{(1)} \cup C_{(1)} \cup (U_A\setminus\{a\})_{(2)}\\
&\\
&= (A\cup C)_{(1)} \cup   \{a,b\}_{(1)} \cup    (V\setminus\{a\})_{(2)}
\end{array}
$$
is a facet of $\Gamma^2_G$ of leaf type.

 \smallskip
 
 \item{\it star type:} $\tau =\{a^1,\ldots,a^t,b\}_{(1)} \cup C_{(1)} \cup (U_A\setminus\{b\})_{(2)}$ where 
 $G_{\{a^1,\ldots,a^t,b\}}$ is a star of $G_{U_A}$ centered at  $b \in U_A$, and 
 \begin{equation}\label{e:star0}
 C\in \A\big ( G_{U_A} \setminus N_{G_{U_A}}[\{a^1,\ldots,a^t,b\}]  \big ).
 \end{equation}
  Again by construction $\{a^1,\ldots,a^t,b\}$ is a star of $G$, and 
   by Claim~1, $A\cup C \in  \A(G \setminus N_{G_{U_A}}[\{a^1,\ldots,a^t,b\}])$. Then 
$$ 
\begin{array}{ll}
\sigma \cup \tau& 
= A_{(1)} \cup (V \setminus U_A)_{(2)}  \cup \{a^1,\ldots,a^t,b\}_{(1)} \cup C_{(1)} \cup (U_A\setminus\{b\})_{(2)}\\
&\\
&=(A\cup C)_{(1)} \cup \{a^1,\ldots,a^t,b\}_{(1)}  \cup    (V\setminus\{b\})_{(2)}
\end{array}$$
is contained in a facet of $\Gamma^2_G$ of star type.
\end{enumerate}
As demonstrated above, $\sigma \cup \tau$ is always a facet of $\Gamma^2_G$.  Hence $\Gamma^2_{G_{U_A}} \subset  {\lk}_{\Gamma^2_G}(\sigma)$, as required. 
\hfill {(\it End of Claim~2.)}

\smallskip

Finally, \cref{l:AAAAA} guarantees that
${\lk}_{\Gamma^2_G}(\sigma) = \Gamma^2_{G_{U_A}}$ is not Cohen--Macaulay, as desired.
\end{proof}

Let $G$ be a finite graph on the vertices $x_1, \ldots, x_n$.  The {\bf whiskered graph}~\cite{V} based on $G$ is the finite graph on $x_1, \ldots, x_n, y_1, \ldots, y_n$ whose edges are the edges of $G$ together with $\{x_1, y_1\}, \ldots, \{x_n, y_n\}$; see \cref{fig:whisker} for an example. Edge ideals of whiskered graphs are always Cohen-Macaulay~\cite{V}.  

\begin{figure}[ht!]
\begin{tabular}{cc}
\begin{tikzpicture}[scale=1]
\tikzstyle{point}=[inner sep=0pt]
\node (a)at (0,1) {};
\node (b) at (-1,0) {};
\node (c) at (1,0) {};
\draw (a.center) -- (b.center);
\draw (a.center) -- (c.center);
\draw (b.center) -- (c.center);
\node [point,label=left:{\small $x_1$}] at (0,1) {};
\node [point,label=left:{\small $x_2$}] at (-1,0) {};
\node [point,label=right:{\small $x_3$}] at (1,0) {};
\draw[black, fill=black] (a) circle(.05);
\draw[black, fill=black] (b) circle(.05);
\draw[black, fill=black] (c) circle(.05);
\end{tikzpicture}
& 
\begin{tikzpicture}[scale=1]
\tikzstyle{point}=[inner sep=0pt]
\node [point,label=left:{\small $x_1$}](a)at (0,1) {};
\node [point,label=left:{\small $x_2$}](b) at (-1,0) {};
\node [point,label=right:{\small $x_3$}](c) at (1,0) {};
\node [point,label=left:{\small $y_1$}](a') at (0,2.5) {};
\node [point,label=left:{\small $y_2$}](b') at (-1,1.5) {};
\node [point,label=right:{\small $y_3$}](c') at (1,1.5) {};

\draw (a.center) -- (b.center);
\draw (a.center) -- (c.center);
\draw (b.center) -- (c.center);

\draw (a.center) -- (a'.center);
\draw (b.center) -- (b'.center);
\draw (c.center) -- (c'.center);

\draw[black, fill=black] (a) circle(.05);
\draw[black, fill=black] (b) circle(.05);
\draw[black, fill=black] (c) circle(.05);
\draw[black, fill=black] (a') circle(.05);
\draw[black, fill=black] (b') circle(.05);
\draw[black, fill=black] (c') circle(.05);
\end{tikzpicture}
\\
Graph $G$
 &  
Whiskering of $G$
\end{tabular}
\caption{}\label{fig:whisker}
\end{figure}

 A direct consequence of  \cref{t:nice} is the following.

\begin{corollary}[{\bf Cohen--Macaulay graphs which do not have Cohen--Macaulay squares}]\label{c:aaaaa}
Let $G$ be a simple graph with more than one edge. Then $I(G)^2$  is not Cohen--Macaulay in the following cases:
\begin{enumerate}
\item if $G$ is a whiskered graph; 
\item if $G$ is a tree;
\item if $G$ is a connected chordal graph; 
\item if $G$ is a connected Cohen--Macaulay bipartite graph.
\end{enumerate}
\end{corollary}

\begin{proof}
Since $G$ has $2$ or more edges, if  $G$ is whiskered, then it must contain at least two whiskers which along with the edge connecting them satisfy the conditions of \cref{t:nice}. Then $I(G)^2$ is not Cohen--Macaulay by \cref{t:nice}.
Every unmixed tree is a whiskered graph~(\cite{V,HHgtm260}), and a connected chordal graph with no triangle (because of \cref{t:pure}) is a tree, so the remaining two statements follow.

Finally, let $G$ be a connected Cohen--Macaulay bipartite graph.  It follows from \cite{HHbipartite} that $G$ comes from a finite connected partially ordered set $P = \{x_1, \ldots, x_n\}$.  Then $V(G) = \{y_1, \ldots, y_n\} \cup \{z_1, \ldots, z_n\}$ and $y_iz_j$ is an edge of $G$ if and only if $x_i \geq x_j$.  Since $G$ has at least two edges, one can assume that $n \geq 2$.  Let $x_{i_0}$ be a minimal element of $P$ and $x_{j_0}$ a maximal element with $x_{i_0} < x_{j_0}$.  Then
\[
z_{i_0}y_{i_0}, y_{i_0}z_{j_0}, z_{j_0}y_{j_0}  
\]
is an induced path of length $3$.  Hence $I(G)^2$ cannot be Cohen--Macaulay.
\end{proof}

A natural question is if one could relax the Cohen--Macaulay condition in \cref{c:aaaaa}(4). 

\begin{question}
    Which connected unmixed bipartite graphs (\cite[pp.~163]{HHgtm260}) have Cohen--Macaulay squares?  
\end{question}

The next example examines some such graphs.

\begin{example}
The complete bipartite $K_{n,n}$ on $2n$ vertices is unmixed with no triangle.  However, $I(K_{n,n})^2$ is not Cohen--Macaulay unless $n=1$.  The complete bipartite $K_{m,n}$ on $m+n$ vertices cannot be unmixed unless $m=n$.  In particular, the complete bipartite $K_{m,n}$ on $m+n$ vertices has a Cohen--Macaulay square if and only if $m=n=1$. 
\end{example}

\begin{example}
If $G$ is a Cameron--Walker graph \cite{CWgraph}, then $I(G)^2$ cannot be Cohen--Macaulay.  In fact, every unmixed Cameron--Walker graph contains a triangle.
\end{example}

\end{document}